\documentclass[12 pt, reqno]{amsart}
\usepackage{amsmath,times,epsfig,amssymb,amsbsy,amscd,amsfonts,amstext,graphicx,color,bm,amsthm}
\usepackage[all]{xy}
\usepackage{graphicx}
\usepackage{hyperref}

\usepackage{tikz,subfigure}
\usepackage{tikz-cd}

\newcommand{\A}{\mathcal{A}}

\newtheorem{Theorem}{Theorem}[section]
\newtheorem{Definition}[Theorem]{Definition}

\newtheorem{Proposition}[Theorem]{Proposition}

\newtheorem{Remark}[Theorem]{Remark}
\newtheorem{Example}[Theorem]{Example}

\DeclareMathOperator{\rk}{rk}

\DeclareMathOperator{\codim}{codim}

\usepackage{tikz-cd}
\usepackage{tikz}
\usetikzlibrary{arrows}
\tikzset{
v/.style={
  circle, draw, inner sep=2pt, minimum size=3pt, fill=black}
}

\begin{document}

\title[Inductive approach to chromatic and characteristic polynomials]{Inductive approach to chromatic and characteristic polynomials}

\begin{abstract} In this article we describe a new inductive approach to compute the chromatic polynomial of simple graphs and the characteristic polynomial of central hyperplane arrangements.
\end{abstract}

\author{Madison Cox}
\address{Madison Cox, Department of Mathematics \& Statistics, Northern Arizona University,
801 S Osborne Drive,
Flagstaff, AZ 86011, USA.}
\email{mfc96@nau.edu}
\author{Michele Torielli}
\address{Michele Torielli, Department of Mathematics \& Statistics, Northern Arizona University,
801 S Osborne Drive,
Flagstaff, AZ 86011, USA.}
\email{michele.torielli@nau.edu}


\date{\today}
\maketitle



\section{Introduction}

Computing the chromatic polynomial of a simple graph or the characteristic polynomial of a central hyperplane arrangement are in general two extremely difficult problems, see \cite{cromatichard, HalpSharir}. 

Another hard problem, related with the computation of characteristic polynomials, is counting the number of chambers of a real hyperplane arrangement. There are several known techniques to compute such invariants, for example using a modular approach \cite{Athanas, PalTorcomb}, and computer algebra systems \cite{BryEbleKuhne, PalTorCoCoA}.

In this article, we describe a new inductive approach to compute the chromatic polynomials of simple graphs and the characteristic polynomials of central hyperplane arrangements, and hence the number of chambers of real arrangements.

This article is organized as follows. In Section 2, we recall the required notions in graph theory and we describe in which situation the chromatic polynomial of a subgraph divides the one of the graph. In Section 3, we generalize the notion of the clique sum graph by introducing the notion of the subgraph sum of two graphs along a common subgraph and we describe how to compute the chromatic polynomial of the new graph from the ones of the starting two graphs. In Section 4, we describe our inductive approach to compute the chromatic polynomial of a simple graph. In Section 5, we recall the required notions on hyperplane arrangements. In Section 6, we recall the notion of the Orlik--Solomon algebra of a hyperplane arrangement and how to use it to describe combinatorial formulas for the coefficients of the characteristic polynomial of a graphic arrangement. In Section 7, we describe our inductive approach to compute the characteristic polynomial and the number of chambers of a graphic arrangements and we use this approach in the case of real $3$-dimensional arrangements. 


\section{Graphs and their chromatic polynomial}

In this paper, we use standard notations for graphs
and we assume that every graph is non-empty, finite, undirected and
simple. We refer to \cite{BondyMurty} as a general reference on the subject.

Consider a graph $G=(V,E)$ and $e\in E$. The \textbf{deletion of the edge $e$} is the graph $G\setminus e=(V,E')$, where $E'=E\setminus\{e\}$, i.e., it is the graph with the same vertex set of $G$ and the edge set is obtained from just deleting the edge $e$. 
 The \textbf{contraction of the edge $e=\{u,v\}$} is the graph $G/e=(V',E')$, where $V'$ is the vertex set $V$ with cardinality $|V|-1$, where vertex $u$ is identified with vertex $v$ in $V$, and $E'=E\setminus\{e\}$ .
 

\begin{Definition}[\cite{read}]
A \textbf{proper coloring} of a graph $G$ is a choice of a color for each vertex of $G$ such that no adjacent vertices
have the same color. If this can be done using $k$ colors, then the graph $G$ is said to be $k$-colorable.
The smallest value of $k$ such that $G$ is $k$-colorable is called the \textbf{chromatic number of $G$} and it is denoted by $\chi(G)$.
\end{Definition}

\begin{Definition}[\cite{read}] Let $G$ be a graph. Its \textbf{chromatic polynomial} $\chi(G,t)$ is the polynomial such that for every $k\ge 0$ we have that $\chi(G,k)$ is the number of $k$-coloring of $G$.
\end{Definition}

\begin{Example}\label{ex:defG0}
Consider $G_0$ the graph in Figure~\ref{Fig:squarediag}. Then $\chi(G_0,t)=t(t-1)(t-2)^2$.
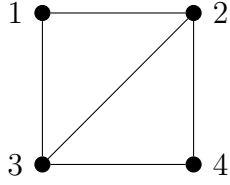
\begin{figure}[t]
\centering
\begin{tikzpicture}
\draw (  0,  2) node[v,label=left:{$ 1 $}](1){};
\draw (2,  2) node[v,label=right:{$ 2 $}](2){};
\draw (  0,  0) node[v,label=left:{$ 3 $}](3){};
\draw (  2,0) node[v,label=right:{$ 4 $}](4){};
\draw (1)--(2);
\draw (1)--(3);
\draw (3)--(4);
\draw (2)--(3);
\draw (2)--(4);
\end{tikzpicture}
\caption{A drawing of the graph $G_0$}\label{Fig:squarediag}
\end{figure}
\end{Example}

\begin{Theorem}[Deletion-contraction formula, \cite{read}]\label{Theo:delcontr}
Let $G=(V,E)$ be a graph. Then for every $e\in E$
$$\chi(G,t)=\chi(G\setminus e,t)-\chi(G/e,t). $$
\end{Theorem}

\begin{Theorem}[\cite{read}]\label{Theo:chromconnectcomp} Let $G$ be a simple graph with connected components $G_1,\dots, G_r$. Then $\chi(G,t)=\prod_{i=1}^r\chi(G_i,t)$.
\end{Theorem}

The following two results are well known, but we write our own proof for completeness.

\begin{Proposition}\label{prop:chartree} Consider $n\ge 1$ and $T_n$ a tree graph with $n$ vertices. Then $$\chi(T_n,t)=t(t-1)^{n-1}.$$
\end{Proposition}
\begin{proof}
We will prove the statement by induction on $n$. If $n=1$, the graph $T_1$ is just an isolated vertex and hence $\chi(T_1,t)=t$. Consider now $n\ge 2$ and $e=\{i,j\}$ an edge of $T_n$ such that $\deg(i)=1$ or $\deg(j)=1$. Then $T_n/e$ is a tree with $n-1$ vertices and $T_n\setminus e$ is a forest with two connected components, one is a tree with $n-1$ vertices and the other is an isolated vertex. Using Theorem~\ref{Theo:delcontr} and Theorem~\ref{Theo:chromconnectcomp}, we have
$$\chi(T_n,t)=\chi(T_n\setminus e,t)-\chi(T_n/e,t)= t^2(t-1)^{n-2}-t(t-1)^{n-2}=t(t-1)^n.$$
\end{proof}

\begin{Proposition}\label{prop:charcomplete} Consider $n\ge 2$ and $K_n$ the complete graph on $n$ vertices. Then $$\chi(K_n,t)=t(t-1)\cdots(t-(n-1)).$$
\end{Proposition}
\begin{proof}
We will prove the statement by induction on $n$. If $n=2$, the graph $K_2$ is just a tree with $2$ vertices and hence, by Proposition~\ref{prop:chartree} $\chi(K_2,t)=t(t-1)$. Consider now $n\ge 3$, $v$ a vertex of $K_n$ and $e_1,\dots,e_{n-1}$ the edges of $K_n$ incident to $v$. Notice that $K_n/e_i$ is a complete graph with $n-1$ vertices. Deleting and contracting the edges $e_1,\dots,e_{n-1}$  one at a time and using Theorem~\ref{Theo:delcontr} and Theorem~\ref{Theo:chromconnectcomp}, we obtain
$$\chi(K_n,t)=\chi(K_n\setminus\{e_1,\dots,e_{n-1}\} ,t)-(n-1)\chi(K_n/e_1,t)=$$ $$=t^2(t-1)\cdots(t-(n-2))-(n-1)t(t-1)\cdots(t-(n-2))=$$ $$=t(t-1)\cdots(t-(n-1)).$$
\end{proof}

There are several known formulas for the chromatic polynomial of cycle graph, see for example \cite{read}. The next result describes the formulation that is more useful for our goal.

\begin{Proposition}\label{prop:charcycle} Consider $n\ge3$ and $C_n$ the cycle graph on $n$ vertices. Then $$\chi(C_n,t)=t(t-1)\sum_{i=0}^{n-2}(-1)^i(t-1)^{n-2-i}.$$
\end{Proposition}
\begin{proof}
We will prove the statement by induction on $n$. If $n=3$, the graph $C_3$ is just a complete graph on $3$ vertices and hence, by Proposition~\ref{prop:charcomplete} $\chi(C_3,t)=t(t-1)(t-2)=t(t-1)(t-1-1)$. Consider now $n\ge 4$ and $e$ and edge of $C_n$. Then $C_n/e$ is a cycle graph with $n-1$ vertices and $C_n\setminus e$ is a tree with $n$ vertices. Using Theorem~\ref{Theo:delcontr}, we have
$$\chi(C_n,t)=\chi(C_n\setminus e,t)-\chi(C_n/e,t)=$$ $$=t(t-1)^{n-1}-t(t-1)\sum_{i=0}^{n-3}(-1)^i(t-1)^{n-3-i}=t(t-1)\sum_{i=0}^{n-2}(-1)^i(t-1)^{n-2-i}.$$
\end{proof}

\begin{Definition} Let $G$ be a simple graph and $H=(V_H,E_H)$ a subgraph of $G$. $H$ is \textbf{path-intersecting} if for every pair $v_i,v_j\in V_H$ of non-adjacent in $H$ vertices there is no path $P=(V_P,E_P)$ in $G$ between $v_i,v_j$ such that $E_H\cap E_P=\emptyset$ and $V_H\cap V_P=\{ v_i,v_j\}$.
\end{Definition}

\begin{Remark} If $G$ is a simple graph and $H$ is a path-intersecting subgraph, then $H$ is induced, isometric and convex. Notice that the opposite implications are all false, see Examples~\ref{Ex:pathclosed}.
\end{Remark}

\begin{Example}\label{Ex:pathclosed}
Consider the graph $G$ in Figure~\ref{Fig:examplechordhomotopy}. If we consider $S=\{1,2,3,4\}$, then $H=G[S]$ is induced but it is not path-intersecting since that path $G[\{2,3,5\}]$ violate our condition. However, this graph is isometric and convex.
\end{Example}


\begin{Theorem}\label{Theo:charinducedsub} Let $G$ be a graph and $H$ a path-intersecting subgraph. Then $\chi(H,t)$ divides $\chi(G,t)$.
\end{Theorem}
\begin{proof}
Assume that $G=(V_G,E_G)$ and $H=(V_H,E_H)$. We will prove the statement by induction on $|E_G\setminus E_H|$. If $E_G=E_H$, then $G$ is obtained from $H$ by adding some isolated vertices (possible none), and hence the statement is true by Theorem~\ref{Theo:chromconnectcomp}. Assume there exists $e\in E_G\setminus E_H$. Then both $G\setminus e$ and $G/e$ have strictly less edges than $G$ and contain $H$ as path-intersecting subgraph. By inductive hypothesis, we then have that $\chi(H,t)$ divides $\chi(G\setminus e,t)$ and $\chi(G/e,t)$. We then conlude by Theorem~\ref{Theo:delcontr}.
\end{proof}

Notice that in Theorem~\ref{Theo:charinducedsub} the assumption that $H$ is path-intersecting cannot be removed.

\begin{Example}\label{Ex:Chordhomotopy}
Consider the graph $G$ in Figure~\ref{Fig:examplechordhomotopy}. If we consider $S=\{1,2,3,4\}$, then $H=G[S]$ is a cycle on $4$ vertices and hence, by Proposition~\ref{prop:charcycle},
we have $\chi(H,t)=t(t-1)(t^2-3t+3)$. Using Theorem~\ref{Theo:delcontr}, we can compute $\chi(G,t)=t(t-1)(t^3-5t^2+10t-7)$. Now $\chi(H,t)$ does not divide $\chi(G,t)$ exactly because there is a path between vertices $2$ and $3$ outside $H$, in fact $H$ is not path-intersecting.
\begin{figure}[t]
\centering
\begin{tikzpicture}
\draw (  0,  2) node[v,label=left:{$ 1 $}](1){};
\draw (2,  2) node[v,label=right:{$ 2 $}](2){};
\draw (  0,  0) node[v,label=left:{$ 3 $}](3){};
\draw (  2,0) node[v,label=right:{$ 4 $}](4){};
\draw (  -1,3) node[v,label=left:{$ 5 $}](5){};
\draw (1)--(2)--(4)--(3);
\draw (1)--(3);
\draw (2)--(5);
\draw (3)--(5);
\end{tikzpicture}
\caption{A drawing of the graph $G$ of Example~\ref{Ex:Chordhomotopy}}\label{Fig:examplechordhomotopy}
\end{figure}
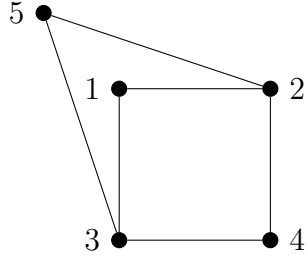
\end{Example}

\begin{Definition} Let $G$ be a graph with $m$ connected components and $v$ a vertex of $G$. We call $v$ a \textbf{cut vertex} if removing $v$ from $G$ gives us a new graph with $s$ connected components such that $s>m$.
\end{Definition}

\begin{Remark} In Theorem~\ref{Theo:charinducedsub}, if we assume that all the vertices of $H$ are cut vertices in $G$ and $H$ is induced, then $H$ is path-intersecting.
\end{Remark}

\begin{Remark}\label{rem:cutvertextopath} In Theorem~\ref{Theo:charinducedsub}, if $H$ is a path of length $m-1$ from $v$ to $w$, then the assumption that $H$ is path-intersecting is equivalent to require that $V_H\setminus\{v, w\}$ are all cut vertices and $H$ is induced.
\end{Remark}

\section{Graph gluing and chromatic polynomials}

In this section we will generalize the notion of the clique sum and the related clique sum Theorem.

Consider two simple graphs $G_1=(V_1,E_1)$ and $G_2=(V_2,E_2)$. Assume there exist $S_1=\{v_1,\dots, v_m\}\subseteq V_1$ and $S_2=\{w_1,\dots, w_m\}\subseteq V_2$ such that $G_1[S_1]\cong G_2[S_2]$. Without loss of generality, we can assume that this isomorphism maps $v_i$ to $w_i$ for all $i=1,\dots, m$. Then we can define 
$V$ to be the vertex set $V_1\cup V_2$ of cardinality $|V_1|+|V_2|-|S_2|$, where the vertices $w_i$ are identified with $v_i$, and $E=E_1\cup E_2$.
With these two sets, we can construct a new graph.
\begin{Definition} The simple graph $G_1\bigoplus_{G_1[S_1]}G_2=(V,E)$ is called the \textbf{subgraph sum graph}.
\end{Definition}

\begin{Remark} In the previous set up, if $G_1[S_1]\cong G_2[S_2]\cong K_m$, then the graph $G_1\bigoplus_{K_m}G_2$ is called a \textbf{clique sum} in the literature and it is usually denoted just by $G_1\bigoplus_{m}G_2$
\end{Remark}

\begin{Theorem}[Clique sum Theorem, \cite{read}]\label{Theo:cliquesum} Consider two graphs $G_1=(V_1,E_1)$ and $G_2=(V_2,E_2)$. Assume there exist $S_1\subseteq V_1$ and $S_2\subseteq V_2$ such that $G_1[S_1]\cong G_2[S_2]\cong K_m$. Then
$$\chi(G_1\bigoplus_{m}G_2,t)=\frac{\chi(G_1,t)\chi(G_2,t)}{\chi(K_m,t)}=\frac{\chi(G_1,t)\chi(G_2,t)}{t(t-1)\cdots(t-(m-1))}. $$
\end{Theorem}

In general, one would like to generalize Theorem~\ref{Theo:cliquesum} by dropping the assumption that $G_1[S_1]\cong G_2[S_2]$ is isomorphic to a complete graph. Unfortunately, this is false in general.

\begin{Example}\label{Ex:nocliquesum}
Consider the graph $G$ in Figure~\ref{Fig:nocliquesum}. $G$ can be see as the subgraph sum of $G_1=G[\{1,2,3,4,5\}]$ and $G_2=G[\{1,2,3,4,6\}]$ along a $C_4$ given by $G[\{1,2,3,4\}]$.
Now the graphs $G_i$ are isomorphic to the graph described in Example~\ref{Ex:Chordhomotopy}, and hence $\chi(G_i,t)=t(t-1)(t^3-5t^2+10t-7)$.
As a consequence the formula $\frac{\chi(G_1,t)\chi(G_2,t)}{\chi(C_4,t)}$ does not even give us a polynomial, since $\chi(G_i,t)$ is not divisible by $\chi(C_4,t)$, as described in Example~\ref{Ex:Chordhomotopy}. On the other hand, $\chi(G,t)=t(t-1)(t^4 -7t^3 +21t^2 -30t +17).$ 
\begin{figure}[t]
\centering
\begin{tikzpicture}
\draw (  0,  2) node[v,label=left:{$ 1 $}](1){};
\draw (2,  2) node[v,label=right:{$ 2 $}](2){};
\draw (  0,  0) node[v,label=left:{$ 3 $}](3){};
\draw (  2,0) node[v,label=right:{$ 4 $}](4){};
\draw (  -1,3) node[v,label=left:{$ 5 $}](5){};
\draw (  3,3) node[v,label=right:{$ 6 $}](6){};
\draw (1)--(2)--(4)--(3);
\draw (1)--(3);
\draw (2)--(5);
\draw (3)--(5);
\draw (1)--(6);
\draw (4)--(6);
\end{tikzpicture}
\caption{A drawing of the graph $G$ of Example~\ref{Ex:nocliquesum}}\label{Fig:nocliquesum}
\end{figure}
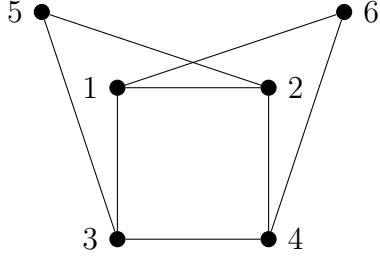
\end{Example}

Using a similar approach as the one used in Theorem~\ref{Theo:charinducedsub}, we have the following.

\begin{Theorem}\label{Theo:sumgraphnopath} Consider two graphs $G_1=(V_1,E_1)$ and $G_2=(V_2,E_2)$. Assume there exist $H_1$ a path-intersecting subgraph of $G_1$ and $H_2$ a path-intersecting subgraph of $G_2$ such that $H_1\cong H_2$. Then
$$\chi(G_1\bigoplus_{H_1}G_2,t)=\frac{\chi(G_1,t)\chi(G_2,t)}{\chi(H_1,t)}. $$
\end{Theorem}
\begin{proof} As described at the beginning of the section, assume $G_1\bigoplus_{H_1}G_2=(V,E)$, where $V=V_1\cup(V_2\setminus V_{H_2})$, and let $H_i=(V_{H_i},E_{H_i})$ for $i=1,2$. We will prove the statement by induction on $s=|V|-|V_{H_1}|$. If $s=0$, then $V=V_{H_1}$ and hence $G=H_1\cong H_2=G_2$. In this case the formula is trivially true.
Assume now that $s\ge 1$. If $E=E_{H_1}$, then $G$ is the graph $H_1$ with some additional isolated vertices, and hence the formula is true. Consider now $e\in E$ such that $e$ is not an edge of $H_1$. Without loss of generalities, we can assume $e\in E_1$. By Theorem~\ref{Theo:delcontr}, 
$$\chi(G,t)=\chi(G\setminus e,t)-\chi(G/e,t). $$
Now $G\setminus e$ is the subgraph sum of the graph $G_1\setminus e$ and $G_2$ along the graph $(G_1\setminus e)[V_{H_1}]=H_1$, and similarly $G/ e$ is the subgraph sum of the graph $G_1/ e$ and $G_2$ along $(G_1/ e)[V_{H_1}]=H_1$. Notice that the quality $(G_1/ e)[V_{H_1}]=H_1$ is ensured by the assumption that $H_1$ is path-intersecting. Moreover, both $G\setminus e$ and $G/ e$ have less edges than $G$ and hence they satisfy the inductive assumption. As a consequence we have
$$\chi(G,t)=\chi(G\setminus e,t)-\chi(G/e,t)=$$ $$=\frac{\chi(G_1\setminus e,t)\chi(G_2,t)}{\chi(H_1,t)}-\frac{\chi(G_1/e,t)\chi(G_2,t)}{\chi(H_1,t)}=\frac{\chi(G_1,t)\chi(G_2,t)}{\chi(H_1,t)}.$$

\end{proof}

\begin{Remark} In Theorem~\ref{Theo:sumgraphnopath}, if we assume that $G_1[S_1]\cong G_2[S_2]\cong K_m$, then the assumption on non-adjacent vertices is trivially satisfied, and hence we get back the clique sum Theorem~\ref{Theo:cliquesum}. 
\end{Remark}

\section{Inductive approach to chromatic polynomials}

Consider $G=(V,E)$ a simple graph, $v_1,v_2\in V$ such that $e_0=\{v_1,v_2\}\notin E$ and $G'=(V,E')$, where $E'=E\cup e_0$. 
The goal of this section is to describe how to compute $\chi(G',t)$ from $\chi(G,t)$.

\begin{Theorem}\label{Theo:charnonewcycles} Assume $e_0$ does not belong to any cycles of $G'$. Then $$\chi(G',t)=\frac{\chi(G,t)(t-1)}{t}.$$
\end{Theorem}
\begin{proof} The assumption that $e_0$ does not belong to any cycles of $G'$ implies that $e_0$ is the bridge between two connected components $G_1,G_2$ of $G$. This implies that $G'\setminus e_0=G$ and $G'/ e_0=\tilde{G_1}\bigoplus_1 \tilde{G_2}$, where $\tilde{G_1}$ and $\tilde{G_2}$ are two subgraphs of $G$ containing $G_1$ and $G_2$ respectively.
By Theorem~\ref{Theo:delcontr} we obtain
$$\chi(G',t)=\chi(G'\setminus e_0,t)-\chi(G'/e_0,t)=\chi(G,t)-\frac{\chi(\tilde{G_1},t)\chi(\tilde{G_2},t)}{t}= $$
$$=\chi(G,t)-\frac{\chi(G,t)}{t}=\frac{\chi(G,t)(t-1)}{t}.$$

\end{proof}

\begin{Theorem}\label{Theo:addoneminimalcycle} Assume there exists $H=(V_H,E_H)$ a subgraph of $G'$ such that $V_H=\{v_1,v_2,\dots,v_m\}$, $E_H=\{e_0,\{v_2,v_3\},\dots\{v_{m-1},v_m\},\{v_1,v_m\}\}$, $H\cong C_m$, for some $m\ge3$ and 
the vertices $v_3,\dots, v_m$ are cut vertices of $G$. Then
 $$\chi(G',t)=\frac{\chi(G,t)\chi(C_m,t)}{t(t-1)^{m-1}}=\frac{\chi(G,t)\sum_{i=0}^{m-2}(-1)^i(t-1)^{m-2-i}}{(t-1)^{m-2}}.$$
 
\end{Theorem}

\begin{proof} We will prove the statement on induction on $m\ge 3$.
Assume $H\cong C_3=K_3$. In this situation $V_H=\{v_1,v_2,v_3\}$ and $v_3$ is a cut vertex of $G$. As a consequence, $G'\setminus e_0=G$ can be seen as the clique sum of two graphs $G_1$ and $G_2$ at a $1$-clique, i.e., at the vertex $v_3$, and similarly, $G'/e_0$ can be seen as the clique sum of the two graphs $G_1$ and $G_2$ at a $2$-clique, i.e., at the edge $\{v_1,v_3\}$. By Theorem~\ref{Theo:delcontr} and Theorem~\ref{Theo:sumgraphnopath}, we obtain
$$\chi(G',t)=\chi(G'\setminus e_0,t)-\chi(G'/e_0,t)=\frac{\chi(G_1,t)\chi(G_2,t)}{t}-\frac{\chi(G_1,t)\chi(G_2,t)}{t(t-1)}=$$
$$=\frac{\chi(G_1,t)\chi(G_2,t)(t-2)}{t(t-1)}=\frac{\chi(G,t)(t-2)}{(t-1)}.$$
Assume now that $m\ge 3$. In this situation, we have that $G'\setminus e_0=G$ can be seen as the subgraph sum of two graphs $G_1$ and $G_2$ at a path on $m-2$ vertices, i.e., at the vertices $v_3,\dots, v_m$, such that $v_i\in G_i$, for $i=1,2$. Similarly, $G'/e_0$ can be seen as the subgraph sum of the two graphs $\tilde{G_1}$ and $\tilde{G_2}$ at $H/e_0\cong C_{m-1}$, where $\tilde{G_1}$ is the graph obtained from $G_1$ adding the edge $\{v_1,v_m\}$ and $\tilde{G_2}$ is the graph obtained from $G_2$ adding the edge $\{v_2,v_m\}$. Notice that the graph $\tilde{G_1}$, respectively $\tilde{G_2}$, is obtained from the graph $G_1$, respectively $G_2$, adding an edge that creates a minimal cycle $C_{m-1}$. This implies that, by inductive hypotheis, $\chi(\tilde{G_i},t)=\frac{\chi(G_i,t)\chi(C_{m-1},t)}{t(t-1)^{m-2}}$, for $i=1,2$. By Theorem~\ref{Theo:delcontr} and Theorem~\ref{Theo:sumgraphnopath}, we obtain
$$\chi(G',t)=\chi(G'\setminus e_0,t)-\chi(G'/e_0,t)=\frac{\chi(G_1,t)\chi(G_2,t)}{t(t-1)^{m-3}}-\frac{\chi(\tilde{G_1},t)\chi(\tilde{G_2},t)}{\chi(C_{m-1},t)}=$$
$$\frac{\chi(G_1,t)\chi(G_2,t)}{t(t-1)^{m-3}}-\frac{\chi(G_1,t)\chi(G_2,t)\chi(C_{m-1},t)^2}{\chi(C_{m-1},t)t^2(t-1)^{2m-4}}=$$
$$\frac{\chi(G_1,t)\chi(G_2,t)}{t(t-1)^{m-3}}-\frac{\chi(G_1,t)\chi(G_2,t)\sum_{i=0}^{m-3}(-1)^i(t-1)^{m-3-i}}{t(t-1)^{m-3}(t-1)^{m-2}}=$$
$$=\frac{\chi(G,t)\sum_{i=0}^{m-2}(-1)^i(t-1)^{m-2-i}}{(t-1)^{m-2}}=\frac{\chi(G,t)\chi(C_m,t)}{t(t-1)^{m-1}}.$$

\end{proof}

\begin{Remark}[cf. Remark~\ref{rem:cutvertextopath}] The assumption in Theorem~\ref{Theo:addoneminimalcycle} that the vertices $v_3,\dots, v_m$ are cut vertices of $G$ is equivalent to require that $H$ is path-intersecting.
\end{Remark}

\begin{Remark} By the previous Remark and Theorem~\ref{Theo:charinducedsub}, the assumptions in Theorem~\ref{Theo:addoneminimalcycle} that the vertices $v_3,\dots, v_m$ are cut vertices of $G$ and that $H\cong C_m$ imply that $t(t-1)^{m-1}$ divides $\chi(G,t)$, and hence the formula described in the Theorem gives us a polynomial.
\end{Remark}

We can now use Theorems~\ref{Theo:charinducedsub}, \ref{Theo:charnonewcycles} and \ref{Theo:addoneminimalcycle}, to describe how to compute the chromatic polynomial of any simple graph in an inductive way.

Let $G=(V,E)$ be a simple graph with $n=|V|$. By Theorem~\ref{Theo:chromconnectcomp}, we can assume that $G$ is connected and let $T$ be a spanning tree of $G$. By Proposition~\ref{prop:chartree}, 
$\chi(T,t)=t(t-1)^{n-1}$. Now $G$ is obtained from $T$ by adding the $|E|-n+1$ remaining edges. We first add all the edges that create one single new cycle and use Theorem~\ref{Theo:charnonewcycles} to compute the chromatic polynomial of the new graph. Finally, we then add the remaining edges that create two or more cycles at the same time and use Theorem~\ref{Theo:charinducedsub} to compute the chromatic polynomial of the new graph.

\begin{Example}\label{Ex:exindconstruction1} Consider the graph $G$ on the left in Figure~\ref{Fig:exindconstruction1}. We can compute its chromatic polynomial using our inductive approach. Consider its spanning tree $T$ on the right in Figure~\ref{Fig:exindconstruction1}. The chromatic polynomial of $T$ is $\chi(T,t)=t(t-1)^4$. We now obtain $G$ by adding first the edge $\{3,5\}$, then $\{1,4\}$ and finally $\{2,4\}$. In this way, in each step we add one single $3$-cycle and hence we change the chromatic polynomial three times by multiplying a factor of $\frac{t-2}{t-1}$, obtaining that $\chi(G,t)=t(t-1)(t-2)^3$.

Notice that we could decide to follow a different order. For example, we could first add the edge $\{2,4\}$ that creates one $4$-cycle and obtained a new chromatic polynomial $t(t-1)^2(t^2-3t+3)$. Then we could add the edge $\{1,4\}$ that divides the $4$-cycle into two $3$-cycles and hence obtaining the chromatic polynomial $t(t-1)^2(t-2)^2$. Finally, adding the remaining edge gives us the chromatic polynomial we already computed before, since the edge $\{3,5\}$ add only one $3$-cycle.
\begin{figure}[t]
\centering
\begin{tikzpicture}
\draw (  0,  2) node[v,label=left:{$ 1 $}](1){};
\draw (2,  2) node[v,label=right:{$ 2 $}](2){};
\draw (  0,  0) node[v,label=left:{$ 3 $}](3){};
\draw (  2,0) node[v,label=right:{$ 4 $}](4){};
\draw (  -1,1) node[v,label=left:{$ 5 $}](5){};
\draw (1)--(2)--(4)--(3)--(1);
\draw (1)--(4);
\draw (1)--(5);
\draw (3)--(5);
\end{tikzpicture}
\hspace{10mm}
\begin{tikzpicture}
\draw (  0,  2) node[v,label=left:{$ 1 $}](1){};
\draw (2,  2) node[v,label=right:{$ 2 $}](2){};
\draw (  0,  0) node[v,label=left:{$ 3 $}](3){};
\draw (  2,0) node[v,label=right:{$ 4 $}](4){};
\draw (  -1,1) node[v,label=left:{$ 5 $}](5){};
\draw (1)--(2);
\draw (1)--(3);
\draw (1)--(5);
\draw (3)--(4);
\end{tikzpicture}
\caption{A drawing on the left of the graph $G$ of Example~\ref{Ex:exindconstruction1} with one of its spanning tree on the right}\label{Fig:exindconstruction1}
\end{figure}
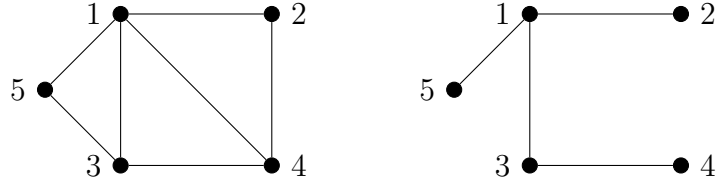
\end{Example}

\section{Hyperplane arrangements}

Let $K$ be a field. A finite set of affine hyperplanes $\A =\{H_1, \dots, H_m\}$ in $K^n$ is called a \textbf{hyperplane arrangement}. For more details on hyperplane arrangements, see \cite{orlterao, DimcaIntro}.
For each hyperplane $H_i$ we fix a polynomial $\alpha_i\in S= K[x_1,\dots, x_n]$ such that $H_i = \alpha_i^{-1}(0)$, 
and let $Q(\A)=\prod_{i=1}^n\alpha_i$. An arrangement $\A$ is called \textbf{central} if each $H_i$ contains the origin of $K^n$. 
In this case, each $\alpha_i$ is a linear homogeneous polynomial, and hence $Q(\A)$ is homogeneous of degree $m$. 

Define the \textbf{lattice of intersections} of $\A$ by
$$L(\A)=\{\bigcap_{H\in\mathcal{B}}H\ne\emptyset \mid \mathcal{B}\subseteq\A\},$$
where if $\mathcal{B}=\emptyset$, we identify $\bigcap_{H\in\mathcal{B}}H$ with $K^n$.
 We endow $L(\A)$ with a partial order defined by $X\le Y$ if and only if $Y\subseteq X$, for all $X,Y\in L(\A)$. 
Note that this is the reverse inclusion. Define a rank function on $L(\A)$ by $\rk(X)=\codim(X)$. 
Moreover, we define $\rk(\A)=\codim(\bigcap_{H\in\mathcal{A}}H)$.
$L(\A)$ plays a fundamental role in the study of hyperplane arrangements, in fact it determines the combinatorics of the arrangement.
Let $$L^k(\A)=\{X\in L(\A)~|~\rk(X)=k\},$$ we call $\A$ \textbf{essential} if $L^n(\A)\ne\emptyset$.

Let $\mu\colon L(\A)\to\mathbb{Z}$ be the \textbf{M\"obius function} of $L(\A)$ defined by
$$\mu(X)=
\begin{cases}
      1 & \text{for } X=K^n,\\
      -\sum_{Y<X}\mu(Y) & \text{if } X>K^n.
\end{cases}$$

The \textbf{characteristic polynomial} of $\A$ is $$\chi(\A,t) = \sum_{X\in L(\A)}\mu(X)t^{\dim(X)}.$$ 

\begin{Proposition}\label{prop:charpoly2dim} Let $\A$ be a central hyperplane arrangement in $K^2$ with $m$ hyperplanes. Then
$$\chi(\A,t)= t^2-mt+(m-1).$$
\end{Proposition}
\begin{proof} This follows from the fact that $\A$ is free with exponents $(1,m-1)$, see \cite{orlterao, PalTor} and Terao's factorization Theorem \cite[Theorem 4.137]{orlterao}.
\end{proof}

The importance of the characteristic polynomial in combinatorics is justified by the following result 
from \cite{craporota}, \cite{orliksolomon} and \cite{zaslavsky1975}.

\begin{Theorem}\label{Theo:countingchamb}
We have
\begin{enumerate}
\item If $\A$ is an arrangement in $\mathbb{F}_p^n$, then $|\mathbb{F}^l_p\setminus \bigcup_{H\in\A}H|=\chi(\A, p)$.
\item If $\A$ is an arrangement in $\mathbb{C}^n$, then the topological $i$-th Betti number of the complement is $b_i(\mathbb{C}^n  \setminus \bigcup_{H\in\A}H)=b_i(\A)$.
\item If $\A$ is an arrangement in $\mathbb{R}^n$, then $|\chi(\A,-1)|$ is the number of chambers and $|\chi(\A, 1)|$ is the number of bounded chambers, where a chamber is a connected component of $\mathbb{R}^\ell\setminus \A$.
\end{enumerate}
\end{Theorem}

Following the previous Theorem, if $\A$ is an arrangement in $\mathbb{R}^n$ we will denote by $c(\A)$, respectively $bc(\A)$, the number of its chambers, respectively its bounded chambers.

For any $X\in L(\A)$ define the subarrangement $\A_X$ of $\A$ by 
$$\A_X=\{H\in\A~|~X\subseteq H\}.$$
Similarly, define the \textbf{restriction} of $\A$ to $X$ as the arrangement $\A^X$ in $X$
$$\A^X=\{X\cap H~|~H\in\A\setminus\A_X \text{ and } X\cap H\ne\emptyset\}.$$
Notice that if $H$ is an hyperplane in $\A$, then $H\cong K^{n-1}$. This implies that $A^H\subseteq K^{n-1}$ for any $H\in\A$.

Given a simple graph $G=(V,E)$ on vertex set $ V=\{1, \dots, n \}$, we can associate to it the arrangement
$$\mathcal{A}(G)= \{\{x_{i}-x_{j}=0\} ~|~ \{i,j\}\in E\}\subseteq K^n. $$
The arrangement $ \mathcal{A}(G) $ is called a \textbf{graphic arrangement}.  A graph and its associated graphic arrangement
share a lot of combinatorial invariants. In particular, we have the following.

\begin{Theorem}[\cite{zaslavsky2012}]\label{Theo:charchrompoly} Let $G=(V,E)$ be a simple graph, $e=\{i,j\}\in E$ and $H=\{x_i-x_j=0\}\in\A(G)$. Then 
\begin{enumerate}
\item $\A(G\setminus e)=\A(G)\setminus\{H\}$.
\item $\A(G/ e)=\A(G)^H$.
\item $\chi(G,t)=\chi(\mathcal{A}(G),t)$.
\end{enumerate}
\end{Theorem}

The previous Theorem also implies that the characteristic polynomial of a graphic arrangement satisfies the deletion contraction formula of Theorem~\ref{Theo:delcontr}.
However, there is a more general formulation that we will need later.

\begin{Theorem}[\cite{orlterao}]\label{Theo:delrestrarr} Let $\A$ be an arrangement in $K^n$ and $H$ a hyperplane of $\A$. Then
$$\chi(\A,t)=\chi(\A\setminus H,t)-\chi(\A^H,t).$$
\end{Theorem}

\section{Orlik--Solomon algebra}
Let $\A=\{H_1, \dots, H_m\}$ be an arrangement of hyperplanes in $\mathbb{C}^n$.
Let $E^1=\bigoplus_{j=1}^m\mathbb{C} e_j$ be the free module generated by $e_1, e_2, \dots, e_m$, where $e_i$ is a symbol corresponding to the hyperplane $H_i$.
Let $E=\bigwedge E^1$ be the exterior algebra over $\mathbb{C}$. The algebra $E$ is graded via $E=\bigoplus_{p=0}^mE^p$, where $E^p=\bigwedge^pE^1$.
The $\mathbb{C}$-module $E^p$ is free and has the distinguished basis consisting of monomials $e_S:=e_{i_1}\wedge\cdots\wedge e_{i_p}$,
where $S=\{{i_1},\dots, {i_p}\}$ runs through all the subsets of $\{1,\dots,m\}$ of cardinality $p$ and $i_1<i_2<\cdots<i_p$.
The graded algebra $E$ is a commutative DGA with respect to the differential $\partial$ of degree $-1$ uniquely defined
by the conditions $\partial e_i=1$ for all $i=1,\dots, m$ and the graded Leibniz formula. Then for every $S\subseteq\{1,\dots,m\}$ of cardinality $p$
$$\partial e_S=\sum_{j=1}^p(-1)^{j-1}e_{S_j},$$
where $S_j$ is the complement in $S$ to its $j$-th element.

For every $S\subseteq\{1,\dots,m\}$, put $\cap S=\bigcap_{i\in S}H_i$ (possibly $\cap S=\emptyset$). 
The subset $S\subseteq\{1,\dots,m\}$ is called \textbf{dependent} if $\cap S\ne\emptyset$
and the set of linear polynomials $\{\alpha_i~|~i\in S\}$ with $H_i=\alpha_i^{-1}(0)$, is linearly dependent.
\begin{Definition}\label{def:osalgbr}
The \textbf{Orlik--Solomon ideal} of $\A$ is the ideal $I=I(\A)$ of $E$ generated by
\begin{enumerate}
\item all $e_S$ with $\cap S=\emptyset$,
\item all $\partial e_S$ with $S$ dependent.
\end{enumerate}
The algebra $A:=A^\bullet(\A)=E/I(\A)$ is called the \textbf{Orlik--Solomon algebra} of $\A$.
\end{Definition}
Clearly $I$ is a homogeneous ideal of $E$ and $I^p=I\cap E^p$ whence $A$ is a graded algebra and we can write $A=\bigoplus_{p\ge 0} A^p$, where $A^p=E^p/I^p$.
If $\A$ is central, then for any $S\subseteq\A$, we have $\cap S\neq\emptyset$. Therefore, the Orlik--Solomon ideal is generated
by the elements of type $(2)$ from Definition \ref{def:osalgbr}.
In this case, the map $\partial$ induces a well-defined differential $\partial\colon A^\bullet(\A)\longrightarrow A^{\bullet -1}(\A)$.

Notice that since $n+k$ hyperplanes in $\mathbb{C}^n$ are dependent for any $k\ge 1$, we have that $A^{n+k}=\{0\}$ for any $k\ge 1$.

The Orlik--Solomon algebra gives us an alternative way to compute the characteristic polynomial.
\begin{Theorem}[\cite{orlterao}]\label{Th:oscharpoly} Let $\A$ be an arrangement in $K^n$. Then the coefficient of $t^{n-k}$ in $\chi(\A,t)$ is $(-1)^k\dim_{\mathbb{C}}(A^k)$, for all $k=0,\dots, n$.
\end{Theorem}

\begin{Remark} Let $G$ be a simple graph and $\A(G)$ the associated graphic arrangement. Then there is a bijection between cycle in $G$ and dependent sets of $\A(G)$, and hence generators of the Orlik--Solomon ideal of $\A(G)$.
\end{Remark}

Thanks to this connection and Theorem~\ref{Th:oscharpoly}, we can prove the following combinatorial formulas for the coefficients of the chromatic polynomial of a simple graph. 

\begin{Theorem} Consider $G$ a graph with $n$ vertices and $m$ edges. Then $\chi(G,t)=\chi(\A(G),t)$ has the following properties.
\begin{enumerate}
\item It is a monic polynomial of degree $n$.
\item The coefficient of $t^{n-1}$ is $-m$.
\item The coefficient of $t^{n-2}$ is $\binom{m}{2}-k_3$, where $k_3$ is the number of subgraphs of $G$ isomorphic to $K_3$.
\item The coefficient of $t^{n-3}$ is $-[\binom{m}{3}-(m-2)k_3-c_4-k_4+g_0]$, where $c_4,k_4$ and $g_0$ are the number of subgraphs of $G$ isomorphic respectively to $C_4, K_4$ and $G_0$ (see Example~\ref{ex:defG0}).
\end{enumerate}
\end{Theorem}
\begin{proof}
We have the following 
\begin{enumerate}
\item See \cite{read}. 
\item See \cite{read}. 
\item By Theorem~\ref{Th:oscharpoly}, we only need to compute $\dim_{\mathbb{C}}(A^2)=\dim_{\mathbb{C}}(E^2)-\dim_{\mathbb{C}}(I^2)$. Now $\dim_{\mathbb{C}}(E^2)=\binom{m}{2}$. Moreover, $I^0=I^1=\{0\}$ implies that $I^2$ is generated by $\partial e_S$, where $S$ is dependent and of cardinality $3$. This implies that $\dim_{\mathbb{C}}(I^2)=k_3$.
\item By Theorem~\ref{Th:oscharpoly}, we only need to compute $\dim_{\mathbb{C}}(A^3)=\dim_{\mathbb{C}}(E^3)-\dim_{\mathbb{C}}(I^3)$. Now $\dim_{\mathbb{C}}(E^3)=\binom{m}{3}$. Moreover, $I^3$ is generated by 2 types of elements: $e_iT$, where $T$ is a generator of $I^2$, and $\partial e_S$, where $S$ is dependent and of cardinality $4$. By the previous point, the number of the first type of elements is $(m-2)k_3$, and a direct computation shows that the number of the second type of elements that are independent with respect to the first type is $c_4+k_4-g_0$.
\end{enumerate}
\end{proof}

Notice that the formula (4) in the previous Theorem is slightly different from the one described in \cite{dongkoh}.

\section{Inductive approach to characteristic polynomials and number of chambers}
Consider $\A=\{H_1,\dots, H_m\}$ be an arrangement in $\mathbb{R}^n$ and $H$ a hyperplane in $\mathbb{R}^n$ that does not to belong to $\A$.
This section is devoted to discuss how to compute the characteristic polynomial and the number of chambers of $\A\cup H$ from the one of $\A$.

Consider $G=(V,E)$ a simple graph with $n=|V|$ and $m=|E|$. Let $v_1,v_2\in V$ such that $e_0=\{v_1,v_2\}\notin E$ and $G'=(V,E')$, where $E'=E\cup \{e_0\}$. Associated to $e_0$
we can consider $H_0$ a hyperplane in $\mathbb{R}^n$. With this setup, we have $\A(G')=\A(G)\cup H_0$.
Using the results of Section 4, Theorem~\ref{Theo:countingchamb} and Theorem~\ref{Theo:charchrompoly}, we obtain the following.


\begin{Theorem}\label{Theo:arrnocyclesadded} Assume that if there exists $S\subseteq\{0,1,\dots, m\}$ such that $S$ is dependent for $\A(G')$ and $|S|\le n$, then $0\notin S$, then 
 $$\chi(\A(G'),t)=\frac{\chi(\A(G),t)(t-1)}{t}.$$
Moreover, $c(\A(G'))=2c(\A(G))$.

\end{Theorem}

\begin{Theorem}\label{Theo:arronecyclesadded} Assume there exists $S\subseteq\{0,1,\dots, m\}$ such that $S$ is dependent for $\A(G')$, $0\in S$, and there does not exist $T\subseteq\{0,1,\dots, m\}$ such that $T$ is dependent for $\A(G')$, $0\in T$, $T\ne S$ and $|T|\le|S|$, then 
$$\chi(\A(G'),t)=\frac{\chi(\A(G),t)\sum_{i=0}^{|S|-2}(-1)^i(t-1)^{|S|-2-i}}{(t-1)^{|S|-2}}.$$
Moreover, $c(\A(G'))=(\sum_{i=0}^{|S|-2}\frac{1}{2^i})c(\A(G))=(2-\frac{1}{2^{|S|-2}})c(\A(G))$.

\end{Theorem}

A natural question is to understand if the previous statements are true also for non graphic arrangements. Unfortunately they are not. 
For example, it is enough to notice that every graphic arrangement is not essential and hence its characteristic polynomial is always divisible by $t$.
This implies that Theorem~\ref{Theo:arrnocyclesadded} cannot be true for essential arrangements.
However, we can still obtain similar results.

Given $\A$ a central arrangement in $\mathbb{R}^3$ with $m$ hyperplanes, we can associate to it the corresponding line arrangement $\tilde{\A}$ in $\mathbb{P}^2$, and 
$\chi(\A,t)=\chi(\tilde{\A},t)$. This means that if one wants to consider $H_0$ a hyperplane in $\mathbb{R}^3$ containing the origin that does not belong to $\A$ and wants to compute 
$\chi(\A\cup H_0,t)$, one can do it by considering the corresponding line $\ell_0$ in $\mathbb{P}^2$ that does not belong to $\tilde{\A}$ and compute $\chi(\tilde{\A}\cup\ell_0,t)$.

\begin{Theorem} Let $\tilde{\A}$ be a line arrangement in $\mathbb{P}^2$ and $\ell_0$ a line in $\mathbb{P}^2$ that does not belong to $\tilde{\A}$. If
$$|\ell_0\cap(\bigcup_{\ell\in\tilde{\A}}\ell)|=k,$$
then $$\chi(\tilde{\A}\cup\ell_0,t)=\chi(\tilde{\A},t)-(t^2-kt+(k-1)).$$
Moreover, $c(\A\cup H_0)=c(\A)+2k$.
\end{Theorem}
\begin{proof} The first part of the result follow from using Theorem~\ref{Theo:delrestrarr} and Proposition~\ref{prop:charpoly2dim}. The second part follows from the first one using 
Theorem~\ref{Theo:countingchamb}.
\end{proof}

\paragraph{\textbf{Acknowledgements}} The authors would like to thank D. Ernst, N. Sieben and T. Shuhei for many helpful discussions. During the preparation of this article the second author was supported by Perko Research Award.



\end{document}